\documentclass[11pt]{amsart}

\numberwithin{equation}{section}

\usepackage{fullpage, graphicx}

\def\F{\mathbb F}

\theoremstyle{plain}
\newtheorem{theorem}{Theorem}[section]
\newtheorem{corollary}[theorem]{Corollary}
\newtheorem{proposition}[theorem]{Proposition}
\newtheorem{lemma}[theorem]{Lemma}

\theoremstyle{definition}

\newtheorem{remark}[theorem]{Remark}

\begin{document}

\title[Code loops]{Code loops in dimension at most $8$}

\author[O'Brien]{E.A.\ O'Brien}

\author[Vojt\v{e}chovsk\'y]{Petr Vojt\v{e}chovsk\'y}

\address[O'Brien]{Department of Mathematics, University of Auckland, Private Bag 92019, Auckland, New Zealand}

\email{e.obrien@auckland.ac.nz}

\address[Vojt\v{e}chovsk\'y]{Department of Mathematics, University of Denver, 2280 S Vine St, Denver, Colorado 80112, U.S.A.}

\email{petr@math.du.edu}



\begin{abstract}
Code loops are certain Moufang $2$-loops constructed from doubly even binary codes that play an important role in the construction of local subgroups of sporadic groups. More precisely, code loops are central extensions of the group of order $2$ by an elementary abelian $2$-group $V$ in the variety of loops such that their squaring map,  commutator map and associator map are related by combinatorial polarization and the associator map is a trilinear alternating form.

Using existing classifications of trilinear alternating forms over the field of $2$ elements, we enumerate code loops of dimension $d=\dim(V)\le 8$ (equivalently, of order $2^{d+1}\le 512$) up to isomorphism. There are $767$ code loops of order $128$, and $80826$ of order $256$, and $937791557$ of order $512$.
\end{abstract}

\thanks{O'Brien was supported by the Marsden Fund of New Zealand
via grant UOA 1323.  Vojt\v{e}chovsk\'y was
supported by Simons Foundation Collaboration Grant 210176 and PROF Grant of the University of Denver.}

\maketitle

\section{Introduction}

Code loops are certain Moufang $2$-loops constructed from doubly even binary codes that play an important role in the construction of local subgroups of sporadic groups \cite{Aschbacher, Conway, Griess}.

We enumerate the code loops of order $128$, $256$ and $512$ up to isomorphism, so extending the work of Nagy and Vojt\v{e}chovsk\'y \cite{NagyVojtechovsky64}. The results are summarized in Tables \ref{Tb:Main} and \ref{Tb:CLA}. The code loops can be constructed explicitly; those of orders dividing $256$ will be available in a future release of the \texttt{LOOPS} package \cite{LOOPS} for {\sf GAP} \cite{GAP}.

The theoretical results required for the classification of code loops were
described briefly in \cite{NagyVojtechovsky64}, in the context of a larger project of enumerating all Moufang loops of order $64$. Since our work suggests that it will be difficult to extend the classification of code loops beyond order $512$ (see Remark \ref{Rm:Feasible}), we carefully record the theory here.

In Section \ref{Sc:Preliminaries} we recall the necessary background material on Moufang loops, code loops, symplectic $2$-loops, trilinear alternating forms, combinatorial polarization and small Frattini Moufang loops. In particular, we recall that code loops, symplectic Moufang $2$-loops and small Frattini Moufang $2$-loops are the same objects. The group $GL(V)$ acts naturally on the set $F^V$ of maps $V\to F$ by
\begin{displaymath}
    f\mapsto f^S,\quad f^S(u) = f(uS^{-1}).
\end{displaymath}
We show that two code loops, realised as central extensions of the two-element field $F=\F_2$ by a vector space $V$ over $F$, are isomorphic if and only if their squaring maps $x\mapsto x^2$ (which can be realised as maps $V\to F$) lie in the same orbit of this action.

For $f\in F^V$ with $f(0)=0$, the $m$th derived form $f_m$ of $f$ is the symmetric map $V^m\to F$ defined by
\begin{equation}\label{Eq:DerivedForm}
    f_m(v_1,\dots,v_m) = \sum_{\emptyset\ne I\subseteq\{1,\dots,m\}} (-1)^{m-|I|}f(\sum_{i\in I}v_i).
\end{equation}
If $P:V\to F$ is the squaring map of a code loop $Q$, then $P_2 = C:V^2\to F$ is the commutator map of $Q$, $P_3=A:V^3\to F$ is the associator map of $Q$, and $P_4=0$. Let
\begin{displaymath}
    F^V_4 = \{f\in F^V\ |\ f(0)=0,\,f_4=0\},
\end{displaymath}
so that $F^V_4$ consists of maps $f:V\to F$ such that $f_3$ is a trilinear alternating form. The results of Section \ref{Sc:Preliminaries} imply that $f\in F^V$ is the squaring map of a code loop if and only if $f\in F^V_4$.

In Section \ref{Sc:Action} we therefore study the action of $GL(V)$ on $F^V$ restricted to $F^V_4$, whose
orbits are in one-to-one correspondence with code loops of order $n=2^{\dim(V)+1}$ up to isomorphism.
Suppose that $V$ has ordered basis $(e_1,\dots,e_d)$. A map $f\in F^V_4$ is uniquely determined by the values
\begin{equation}\label{Eq:Values}
    \omega_{i_1\cdots i_k} = f_k(e_{i_1},\dots,e_{i_k})\in F,
\end{equation}
where $1\le k\le 3$ and $1\le i_1<\cdots<i_k\le d$. Conversely, given arbitrary parameters $\omega_{i_1\cdots i_k}\in F$ for every $1\le k\le 3$ and $1\le i_1<\cdots<i_k\le d$, there is a unique map $f\in F^V_4$ such that \eqref{Eq:Values} holds.
The action of $GL(V)$ on $F^V_4$ is therefore equivalent to a certain action of $GL(V)$ on the parameter space
\begin{displaymath}
    \Omega_d = F^{\binom{d}{1}+\binom{d}{2}+\binom{d}{3}}.
\end{displaymath}
The action of $GL(V)$ on $\Omega_d$ is stratified in a sense defined in Section \ref{Ss:Stratified}, and its orbits can thus be calculated in three steps as in Corollary \ref{Cr:Main}, the result on which our calculations are based.

The first step is a linear action of $G=GL(V)$ on $F^{\binom{d}{3}}$, which is equivalent to the classification of trilinear alternating forms over $F$ in dimension $d$. It is not hard to calculate orbit representatives when $d\le 6$. For $d\ge 7$, we use the existing classifications of \cite{CohenHelminck} ($d=7$) and \cite{HoraPudlak} ($d=8$), where orbit representatives and the orders of their automorphism groups are given. We explicitly calculate the automorphism groups, so verifying these results. The trilinear alternating forms
$A:V^3\to F$ correspond to the associator maps in code loops.
Given a trilinear alternating form $A$, the second step amounts to
understanding the orbits and stabilizers of affine actions of
the stabilizer $G_A$ on $F^{\binom{d}{2}}$.
The resulting symmetric alternating maps $C:V^2\to F$ correspond to the
commutator maps in code loops.
Given a map $C$ from the second step, the third step consists of
calculating orbits of affine actions of the stabilizer $G_{A,C}$ on
$F^{\binom{d}{1}}$. The resulting maps $P:V\to F$ correspond to the
squaring maps in code loops.

Our results and information about the calculations are recorded
in Section \ref{Sc:Results}.

\section{Preliminaries}\label{Sc:Preliminaries}

\subsection{Loops and Moufang loops}

A nonempty set $Q$ with a binary operation $\cdot$ and an element $1$ is a \emph{loop} if $x\cdot 1 = 1\cdot x = x$ for every $x\in Q$ and if the translations $L_x:Q\to Q$, $L_x(y)=x\cdot y$ and $R_x:Q\to Q$, $R_x(y)=y\cdot x$ are bijections of $Q$ for every $x\in Q$. We often write $xy$ instead of $x\cdot y$.

A loop $Q$ is \emph{Moufang} if it satisfies the identity $x(y(xz)) = ((xy)x)z$. Moufang loops form a variety of loops with properties close to groups. For instance, any two elements of a Moufang loop generate a group, and if $x(yz)=(xy)z$ for some elements $x$, $y$, $z$ of a Moufang loop, then the subloop generated by $x$, $y$, $z$ is a group \cite{Moufang}. As a significant nonassociative example, nonzero octonions under multiplication form a Moufang loop.

For a prime $p$, a \emph{$p$-loop} is a loop whose order is a power of $p$.

\subsection{Code loops}

Code loops were introduced in 1986 by Griess \cite{Griess} as follows.

Let $F=\F_2$. Given $u=(u_1,\dots,u_d)\in F^d$, let $|u|=\sum_{i=1}^d u_i\in\mathbb Z$ be the Hamming weight of $u$. Given $u=(u_1,\dots,u_d)$ and $v=(v_1,\dots,v_d)\in F^d$, let $u\cap v = (u_1v_1,\dots,u_dv_d)$, where $u_iv_i=1$ if and only if $u_i=1=v_i$.

Let $V$ be a doubly even binary code, that is, a linear subspace of $F^d$ such that $|u|\equiv 0\pmod 4$ for every $u\in V$. Note that then $|u|/4$, $|u\cap v|/2$ and $|u\cap v\cap w|$ are integers for every $u$, $v$, $w\in V$. A function $\theta:V^2\to F$ is a \emph{factor set} if
\begin{align*}
    &\theta(u,u)\equiv |u|/4\pmod 2,\\
    &\theta(u,v)+\theta(v,u)\equiv |u\cap v|/2\pmod 2,\\
    &\theta(u,v)+\theta(u+v,w)+\theta(v,w)+\theta(u,v+w)\equiv |u\cap v\cap w|\pmod 2,
\end{align*}
for all $u$, $v$, $w\in V$.

A \emph{code loop} $Q$ is a loop defined on $F\times V$ by
\begin{displaymath}
    (a,u)(b,v) = (a+b+\theta(u,v),u+v),
\end{displaymath}
where $V$ is a doubly even binary code and $\theta:V^2\to F$ is a factor set.

Griess \cite{Griess} proved that every doubly even code admits a factor set, the isomorphism type of a code loop over $V$ is independent of the choice of the factor set $\theta$, and every code loop is Moufang. Furthermore, he showed that code loops correspond to a certain class of loops considered by Parker (see \cite[Definition 13, Theorem 14]{Griess}), which were in turn used by Conway as one of the key steps in the construction of the Monster sporadic group \cite{Conway}. In the language of code loops, the Parker loop for the Monster group is the code loop associated with the extended binary Golay code.

\subsection{Symplectic $2$-loops}

The connection between sporadic groups and Moufang $2$-loops was reinforced by Aschbacher \cite[Chapters 4 and 10]{Aschbacher}. To explain his results on loops, we first introduce central extensions of loops and symplectic $2$-loops.

The \emph{center} of a loop $Q$ consists of all $x\in Q$ such that $xy=yx$, $x(yz)=(xy)z$, $y(xz)=(yx)z$ and $y(zx)=(yz)x$ for every $y$, $z\in Q$. The center $Z(Q)$ is a normal subloop of $Q$ (that is, a kernel of a loop homomorphism).

A loop $Q$ is a \emph{central extension} of an abelian group $(Z,+)$ by a loop $(V,+)$ if $Z\le Z(Q)$ and $Q/Z$ is isomorphic to $V$. Up to isomorphism, all central extensions of $Z$ by $V$ are obtained as loops $Q(Z,V,\theta)$ defined on $Z\times V$ with multiplication
\begin{displaymath}
    (a,u)(b,v) = (a+b+\theta(u,v),u+v),
\end{displaymath}
where $\theta:V^2\to Z$ is a map satisfying $\theta(0,u)=\theta(u,0)=0$ for every $u\in V$.

The \emph{commutator} $C(x,y)$ of $x, y\in Q$ is the unique element of $Q$ such that $xy = (yx)C(x,y)$, and the \emph{associator} $A(x,y,z)$ of $x, y, z \in Q$ is the
unique element of $Q$ such that $(xy)z = (x(yz))A(x,y,z)$.

Following \cite{Aschbacher}, a loop $Q$ is
a \emph{symplectic $2$-loop} if it is a central extension of the $2$-element group $(Z,+)=(\F_2,+)$ by an elementary abelian $2$-group $(V,+)$. Direct calculation shows that the squaring map $P:Q\to Q$, $x\mapsto x^2$, the commutator map $C:Q^2\to Q$, $(x,y)\mapsto C(x,y)$, and the associator map $A:Q^3\to Q$, $(x,y,z)\mapsto A(x,y,z)$ are given by
\begin{align*}
    P(a,u) &= (\theta(u,u),0),\\
    C((a,u),(b,v)) &= (\theta(u,v)+\theta(v,u),0),\\
    A((a,u),(b,v),(c,w)) &= (\theta(u,v)+\theta(u+v,w)+\theta(v,w)+\theta(u,v+w),0).
\end{align*}
All three maps are therefore well-defined modulo $Z$, and can be viewed as maps $P:V\to \F_2$, $C:V^2\to \F_2$, $A:V^3\to \F_2$.

Comparing this with Griess' definition of a factor set, we see that code loops are symplectic $2$-loops over a doubly even binary code $(V,+)$ in which the squaring map, the commutator map and the associator map are governed by natural intersection properties of codewords of $V$.

\subsection{Trilinear alternating forms}

Let $V$ be a vector space over a field $F$. A map $f:V^m\to F$ is \emph{symmetric} if $f(v_1,\dots,v_m) = f(v_{\pi(1)},\dots,v_{\pi(m)})$ for every $v_1$, $\dots$, $v_m\in V$ and every $\pi\in S_m$; it is \emph{$m$-linear} if it is linear in every coordinate; and it is \emph{alternating} if $f(v_1,\dots,v_m)=0$ whenever $v_i=v_j$ for some $1\le i<j\le m$.
If $F$ has characteristic 2, then it is easily seen that every $m$-linear alternating form is symmetric.

Two $m$-linear alternating forms $f$, $g:V^m\to F$ are \emph{equivalent} if there is $S\in GL(V)$ such that $f(v_1,\dots,v_m) = g(v_1S,\dots,v_mS)$ for every $v_1$, $\dots$, $v_m\in V$.
Trilinear alternating forms over $\F_2$ have been classified up to equivalence in dimensions $d=\dim(V)\le 8$; see \cite{CohenHelminck} for $d\le 7$ and \cite{HoraPudlak} for $d=8$.

Suppose $(e_1,\dots,e_d)$ is an ordered basis of $V$. A trilinear alternating form $f:V^3\to \F_2$ is determined by the values $f(e_{i_1},e_{i_2},e_{i_3})$ with $1\le i_1<i_2<i_3\le d$. As usual, we therefore represent alternating forms in a compact notation that
indicates
for which triples $(e_{i_1},e_{i_2},e_{i_3})$ the form does not vanish.
For instance, $f=123+145$ is the form $f:V^3\to\F_2$ such that $f(e_{i_1},e_{i_2},e_{i_3})=1$ if and only if $\{i_1,i_2,i_3\}\in\{\{1,2,3\}$, $\{1,4,5\}\}$.

\subsection{Combinatorial polarization}

To identify Moufang loops among symplectic $2$-loops, we need the notion of combinatorial polarization due to Ward \cite{Ward}. We follow the terminology and notation of \cite{Aschbacher}.

Let $F$ be a field, $V$ a vector space over $F$, and $f:V\to F$ a map satisfying $f(0)=0$. For $m\ge 1$, the \emph{$m$th derived form} of $f$ is the map $f_m:V^m\to F$ defined by \eqref{Eq:DerivedForm}. Note that \eqref{Eq:DerivedForm} is an analog of the principle of inclusion and exclusion for maps. For instance,
\begin{displaymath}
    f_3(v_1,v_2,v_3) = f(v_1+v_2+v_3) - f(v_1+v_2) - f(v_1+v_3) - f(v_2+v_3) + f(v_1) + f(v_2) + f(v_3).
\end{displaymath}
Further note that $f_1=f$ and every $f_m$ is symmetric.
The maps $f_1$, $f_2$, $\dots$ are \emph{related by polarization}.
It is not difficult to show that if $f_1$, $f_2$, $\dots$ are related by
polarization, then
\begin{equation}\label{Eq:Inductive}
    f_{m+1}(v_1,\dots,v_{m+1}) =
    f_m(v_1+v_2,v_3,\dots,v_{m+1}) {-} f_m(v_1,v_3,\dots,v_{m+1}) {-} f_m(v_2,v_3,\dots,v_{m+1})
\end{equation}
for every $m\ge 1$ and $v_1$, $\dots$, $v_{m+1}\in V$.
Conversely, if $f_m:V^m\to F$ are symmetric maps satisfying \eqref{Eq:Inductive} for every $m\ge 1$ and $v_1$, $\dots$, $v_{m+1}\in V$, then $f_1$, $f_2$, $\dots$ are related by polarization.

We record additional consequences of \eqref{Eq:Inductive} whenever $f_1$, $f_2$, $\dots$ are related by polarization:
\begin{itemize}
\item $f_{m+1}=0$ if and only if $f_m$ is additive in every coordinate,
\item $f_m(v_1,\dots,v_m)=0$ whenever $v_i=0$ for some $i$,
\item if $F$ has characteristic $2$ then $f_m$ is alternating.
\end{itemize}
Consequently, if $F=\F_2$ then $f_{m+1}=0$ if and only if $f_m$ is an $m$-linear alternating form.

\subsection{Symplectic Moufang $2$-loops}

Let $V$ be a vector space over $F=\F_2$. The group $GL(V)$ acts on the space of maps $f:V\to F$ by
\begin{displaymath}
    f^S(v) = f(vS^{-1}),
\end{displaymath}
where $S\in GL(V)$ and $v\in V$. This is indeed an action since $f^{ST}(v) = f(v(ST)^{-1}) = f(vT^{-1}S^{-1}) = f^S(vT^{-1}) = (f^S)^T(v)$. (See
Section \ref{Ss:RelatedActions} for a discussion of related actions.)

Aschbacher \cite[Lemma 13.1, Lemma 13.5, Theorem 13.7]{Aschbacher}
proved the following.
\begin{itemize}
\item A symplectic $2$-loop $Q = Q(F,V,\theta)$ is Moufang if and only if the maps $P$, $C$, $A$ are related by polarization, that is, $C=P_2$ and $A=P_3$.
\item Given a map $f:V\to F$ such that $f(0)=0$, there is a symplectic Moufang $2$-loop with $P=f$ if and only if $f_4=0$, or, equivalently, if and only if $f_3$ is a trilinear alternating form.
\item Two symplectic Moufang $2$-loops over $V$ are isomorphic if and only if their squaring maps are in the same orbit of the action of $GL(V)$ on $F_4^V$.
\end{itemize}

\subsection{Small Frattini Moufang loops}

Hsu \cite{Hsu} showed that the code loops of Griess are precisely the
symplectic Moufang $2$-loops of Aschbacher.

Let $p$ be a prime and let $Q$ be a Moufang $p$-loop. Glauberman and Wright
\cite{Glauberman, GlaubermanWright}
showed that $Q$ is centrally nilpotent.
Let $\Phi(Q)$ be the \emph{Frattini subloop} of $Q$, consisting of all nongenerators of $Q$. Bruck \cite{Bruck} showed that $\Phi(Q)$ is the smallest normal subloop of $Q$ such that $Q/\Phi(Q)$ is an elementary abelian $p$-group. Following \cite{Hsu}, a Moufang $p$-loop $Q$ is \emph{small Frattini} if $|\Phi(Q)|\in \{1,p\}$, and an associative small Frattini Moufang $p$-loop
is a \emph{small Frattini group}. It was shown in \cite{Hsu} that $\Phi(Q)\le Z(Q)$ in every small Frattini Moufang $p$-loop $Q$. Hence the small Frattini Moufang $2$-loops are precisely the symplectic Moufang $2$-loops. Hsu \cite{Hsu} also proved the following.
\begin{itemize}
\item A nonassociative small Frattini Moufang $p$-loop exists if and only if $p\in\{2,3\}$.
\item Small Frattini Moufang $2$-loops (hence symplectic Moufang $2$-loops) are precisely the code loops.
\end{itemize}

\subsection{Related results}

Chein and Goodaire \cite{CheinGoodaire} used an intricate combinatorial construction to show that code loops are precisely Moufang loops with at most two squares. Their construction is used in \cite{Hsu} and generalized in \cite{Vojtechovsky}.

If $f:V^3\to \F_2$ is a trilinear alternating form, then its \emph{radical} is $\{v_1\in V\ |\ f(v_1,v_2,v_3)=0$ for every $v_2$, $v_3\in V\}$.
The radical and \emph{radical polynomial} are key invariants of trilinear alternating forms \cite{HoraPudlak}. Dr\'apal and Vojt\v{e}chovsk\'y \cite{DrapalVojtechovsky} showed that if a Moufang loop $Q$ has an associator map with trivial radical that is equivalent to an associator map of a code loop, then $Q$ is also a code loop. Moreover, they proved that two code loops with equivalent associators can be obtained from one another by a sequence of two kinds
of constructions, known as cyclic and dihedral modifications.

Hsu \cite{HsuExplicit} presented an iterative construction that builds a code loop from a given doubly even binary code. Nagy \cite{Nagy} presented a global construction of code loops based on groups with triality associated with Moufang loops. His construction can
be used to construct explicitly code loops from the
parameters \eqref{Eq:Values}.

We do not pursue here the connections between code loops and sporadic groups \cite{Aschbacher, Conway, Griess}. In this direction, the notion of a code loop has been extended to odd primes in \cite{Griess27, Richardson}. An attempt to unify the theory of code loops for $p=2$ and $p$ odd has been made in \cite{DrapalVojtechovskyCode}.

\section{The action of $GL(V)$}\label{Sc:Action}

Combining the results of Aschbacher and Hsu, we obtain the following.

\begin{theorem}[\cite{Aschbacher}, \cite{Hsu}]\label{Th:AH}
Let $d\ge 3$, $F=\F_2$, $V=F^d$, and let $F^V_4=\{f:V\to F\ |\ f(0)=0$ and $f_4=0\}$. The code loops of order $2^{d+1}$ up to isomorphism are in one-to-one correspondence with orbits of the action of $GL(V)$ on $F^V_4$ given by $f^S(u) = f(uS^{-1})$.
\end{theorem}

In Section \ref{polar} we show that a map $f\in F^V_4$ can be
described by $\binom{d}{1} + \binom{d}{2} + \binom{d}{3}$ parameters in $F$. In Section \ref{Ss:Action} we exhibit the action of $GL(V)$ on the parameter space $\Omega_d=F^{\binom{d}{1} + \binom{d}{2} + \binom{d}{3}}$ that is equivalent to the action of $GL(V)$ on $F^V_4$. In Section \ref{Ss:Stratified} we show how to iteratively calculate the orbits of $GL(V)$ on $\Omega_d$ by restricting the action to various subspaces of $\Omega_d$.

\subsection{Combinatorial polarization with a fixed basis}\label{polar}

Let $F=\F_2$, and let $V$ be a vector space over $F$ with ordered basis $(e_1,\dots,e_d)$. For $u = \sum_{i=1}^d u_ie_i$ let
\begin{displaymath}
    |u| = \sum_{i=1}^d u_i\in\mathbb Z.
\end{displaymath}

\begin{lemma}\label{Lm:ParamsDetermine}
Let $V$ be a vector space over $F=\F_2$ with
ordered basis $(e_1,\dots,e_d)$. If $f:V\to F$ satisfies $f(0)=0$ and $f_{m+1}=0$, then
$f$ is uniquely determined by the parameters \eqref{Eq:Values} with $1\le k\le m$ and $1\le i_1<i_2<\cdots <i_k\le d$.
\end{lemma}
\begin{proof}
We show by induction on $|u|$ that $f(u)$ is determined for every $u\in V$. If $|u|=0$ then $u=0$ and $f(0)=0$ is given. Suppose that $|u|>0$ and $f(v)$ is determined for every $v\in V$ with $|v|<|u|$. Then there exist
$i_1<\cdots <i_k$ such that
$u = \sum_i u_ie_i = e_{i_1}+\cdots+ e_{i_k}$. By \eqref{Eq:DerivedForm},
\begin{displaymath}
    f(u) = f_k(e_{i_1},\dots,e_{i_k}) + \sum_I f(\sum_{i\in I}u_ie_i)
    = \omega_{i_1\cdots i_k} + \sum_I f(\sum_{i\in I}u_ie_i),
\end{displaymath}
where the summation runs over all nonempty, proper subsets $I$ of $\{i_1,\dots,i_k\}$. By the induction assumption, $f(\sum_{i\in I}u_ie_i)$ is known for each such subset $I$. If $k>m$ then $\omega_{i_1\cdots i_k}=0$. Otherwise $k\le m$ and $\omega_{i_1\cdots i_k}$ is one of the given parameters.
\end{proof}

The following result is essentially \cite[Theorem 8.3]{Hsu} combined with
the remarks therein. It allows us to reconstruct explicitly the symmetric maps $f$, $f_2$ and $f_3$ from the parameters \eqref{Eq:Values} whenever $f_4=0$.

\begin{proposition}\label{Pr:Formulas}
Let $V$ be a vector space over $F=\F_2$ with ordered basis $(e_1,\dots,e_d)$.
If $f:V\to F$ satisfies $f(0)=0$ and $f_4=0$, then
\begin{align}
    f_3&(\sum_i x_ie_i,\sum_j y_je_j, \sum_k z_ke_k)\notag\\
    &= \sum_{i,j,k}x_iy_jz_k \omega_{ijk}\label{Eq:f3}\\
    &= \sum_{i<j<k} (x_iy_jz_k+x_iy_kz_j+x_jy_iz_k+x_jy_kz_i+x_ky_iz_j+x_ky_jz_i)\omega_{ijk}\notag
\end{align}
for every $x_i$, $y_j$, $z_k\in F$,
\begin{align}
    f_2&(\sum_i x_ie_i, \sum_j y_je_j)\notag\\
    &=  \sum_{i,j}x_iy_j\omega_{ij} + \sum_k\sum_{i<j} x_ix_jy_k\omega_{ijk} + \sum_i\sum_{j<k}x_iy_jy_k\omega_{ijk}\label{Eq:f2}\\
    &= \sum_{i<j}(x_iy_j+x_jy_i)\omega_{ij}
    + \sum_{i<j<k}(x_ix_jy_k+x_ix_ky_j+x_jx_ky_i + x_iy_jy_k+x_jy_iy_k+x_ky_iy_j)\omega_{ijk}\notag
\end{align}
for every $x_i$, $y_j\in F$, and
\begin{equation}\label{Eq:f}
    f(\sum_i x_ie_i) = \sum_i x_i\omega_i + \sum_{i<j} x_ix_j\omega_{ij} + \sum_{i<j<k} x_ix_jx_k \omega_{ijk}
\end{equation}
for every $x_i\in F$.
\end{proposition}
\begin{proof}
In both \eqref{Eq:f3} and \eqref{Eq:f2} the second equality follows from the fact that each $f_k$ is symmetric and alternating. We therefore prove only the first equality in each of \eqref{Eq:f3}, \eqref{Eq:f2}, \eqref{Eq:f}.

Since $f_4=0$, the form $f_3$ is trilinear and \eqref{Eq:f3} follows.

We prove \eqref{Eq:f2} by induction on $|x|+|y|$. If $|x|\le 1$ and $|y|\le 1$, the result immediately follows because the sums with $i<j$ and $j<k$ are vacuous, and the first sum involves at most one summand. By symmetry, we can now assume without loss of generality that $|x|\ge 2$ and write $x=x'+x''$, where $|x'|<|x|$ and $|x''|<|x|$. By \eqref{Eq:Inductive}, the inductive assumption and \eqref{Eq:f3},
\begin{align*}
    f_2(\sum_i (x_i'&+x_i'')e_i,\sum_j y_je_j)\\
    &= f_2(\sum_i x_i'e_i,\sum_j y_je_j) + f_2(\sum_i x_i''e_i,\sum_j y_je_j)    + f_3(\sum_i x_i'e_i,\sum_j x_j''e_j,\sum_k y_ke_k)\\
    &= \sum_{i,j}x_i'y_j\omega_{ij} + \sum_k\sum_{i<j} x_i'x_j'y_k\omega_{ijk} + \sum_i\sum_{j<k} x_i'y_jy_k\omega_{ijk}\\
    &\quad+ \sum_{i,j} x_i''y_j\omega_{ij} + \sum_k\sum_{i<j}x_i''x_j''y_k\omega_{ijk} + \sum_i\sum_{j<k} x_i''y_jy_k\omega_{ijk}\\
    &\quad+ \sum_{i,j,k}x_i'x_j''y_k\omega_{ijk}.
\end{align*}
We need to show that this is equal to
\begin{displaymath}
    \sum_{i,j}(x_i'+x_i'')y_j\omega_{ij} + \sum_k\sum_{i<j}(x_i'+x_i'')(x_j'+x_j'')y_k\omega_{ijk}
    + \sum_i\sum_{j<k} (x_i'+x_i'')y_jy_k\omega_{ijk}.
\end{displaymath}
Clearly, the sums involving $\omega_{ij}$ agree, and so do the sums involving $y_jy_k$. Since
\begin{displaymath}
    \sum_{i<j}(x_i'+x_i'')(x_j'+x_j'') = \sum_{i<j}x_i'x_j' + \sum_{i<j}x_i''x_j'' + \sum_{i,j}x_i'x_j'',
\end{displaymath}
the remaining sums also agree.

Finally, we prove \eqref{Eq:f} by induction on $|x|$. There is again nothing to show when $|x|\le 1$, so suppose that $|x|\ge 2$ and let $x'$, $x''$ be such that $x=x'+x''$, $|x'|<|x|$ and $|x''|<|x|$.
By \eqref{Eq:Inductive}, the inductive assumption, \eqref{Eq:f3} and \eqref{Eq:f2},
\begin{align*}
    f(\sum_i(x_i'+x_i'')e_i) &= f(\sum_i x_i'e_i + \sum_i x_i''e_i)\\
    &= f(\sum_i x_i'e_i) + f(\sum_i x_i''e_i) + f_2(\sum_i x_i'e_i, \sum_j x_j''e_j)\\
    &=\sum_i x_i'\omega_i + \sum_{i<j}x_i'x_j'\omega_{ij} + \sum_{i<j<k}x_i'x_j'x_k'\omega_{ijk}\\
    &\quad+\sum_i x_i''\omega_i + \sum_{i<j} x_i''x_j''\omega_{ij} + \sum_{i<j<k}x_i''x_j''x_k''\omega_{ijk}\\
    &\quad+ \sum_{i,j}x_i'x_j''\omega_{ij} + \sum_k\sum_{i<j}x_i'x_j'x_k''\omega_{ijk} + \sum_i\sum_{j<k}x_i'x_j''x_k''\omega_{ijk}.
\end{align*}
We leave as an exercise for the reader to show that this is equal to
\begin{displaymath}
    \sum_i (x_i'+x_i'')\omega_i + \sum_{i<j}(x_i'+x_i'')(x_j'+x_j'')\omega_{ij} + \sum_{i<j<k}(x_i'+x_i'')(x_j'+x_j'')(x_k'+x_k'')\omega_{ijk}.
\qedhere
\end{displaymath}
\end{proof}

We now show the converse of Lemma \ref{Lm:ParamsDetermine}, namely that the parameters \eqref{Eq:Values} determine a map $f\in F_4^V$.

\begin{proposition}\label{Pr:Extension}
Let $V$ be a vector space over $F=\F_2$ with ordered basis $(e_1,\dots,e_d)$. Suppose that the parameters $\omega_{i_1\cdots i_k}$ of \eqref{Eq:Values} are given for every $1\le k\le 3$ and $1\le i_1<\cdots<i_k\le d$. Define $f_3:V^3\to F$, $f_2:V^2\to F$ and $f_1=f:V\to F$ according to \eqref{Eq:f3}, \eqref{Eq:f2} and \eqref{Eq:f}, respectively. Then $f_1$, $f_2$, $f_3$ are related by polarization and $f_4=0$.
\end{proposition}
\begin{proof}
The formulas \eqref{Eq:f3}, \eqref{Eq:f2}, \eqref{Eq:f} produce symmetric maps $f_3$, $f_2$, $f_1$, respectively. To show that $f_1$, $f_2$, $f_3$ are related by polarization, it therefore suffices to check that \eqref{Eq:Inductive} holds. Let $x=\sum_i x_ie_i$, $y=\sum_j y_je_j$ and $z=\sum_k z_ke_k$.

The left hand side of
\begin{displaymath}
    f_3(x,y,z) = f_2(x+y,z)-f_2(x,z)-f_2(y,z)
\end{displaymath}
is equal to
\begin{displaymath}
    f_3(\sum_i x_ie_i,\sum_j y_je_j, \sum_k y_ke_k) = \sum_{i,j,k}x_iy_jz_k\omega_{ijk},
\end{displaymath}
while the right hand side is equal to
\begin{align*}
    f_2(\sum_i(x_i&+y_i)e_i,\sum_j z_je_j) - f_2(\sum_i x_ie_i,\sum_j z_je_j) - f_2(\sum_i y_ie_i,\sum_j z_je_j)\\
    &= \sum_{i,j}(x_i+y_i)z_j\omega_{ij} + \sum_k\sum_{i<j}(x_i+y_i)(x_j+y_j)z_k\omega_{ijk} + \sum_i\sum_{j<k}(x_i+y_i)z_jz_k\omega_{ijk}\\
    &\quad- \sum_{i,j}x_iz_j\omega_{ij} - \sum_k\sum_{i<j}x_ix_jz_k\omega_{ijk} - \sum_i\sum_{j<k}x_iz_jz_k\omega_{ijk}\\
    &\quad-\sum_{i,j}y_iz_j\omega_{ij} - \sum_k\sum_{i<j}y_iy_jz_k\omega_{ijk} - \sum_i\sum_{j<k} y_iz_jz_k\omega_{ijk}.
\end{align*}
In this expression, the summands involving $\omega_{ij}$ cancel, and so do all the summands involving $z_jz_k$. Therefore the right hand side reduces to
\begin{displaymath}
    \sum_k\sum_{i<j}(x_iy_j+y_ix_j)z_k\omega_{ijk},
\end{displaymath}
and  equality follows.

Similarly, in
\begin{displaymath}
    f_2(x,y) = f(x+y)-f(x)-f(y)
\end{displaymath}
the left hand side is equal to
\begin{displaymath}
    \sum_{i,j}x_iy_j\omega_{ij} + \sum_k\sum_{i<j} x_ix_jy_k\omega_{ijk} + \sum_i\sum_{j<k} x_iy_jy_k\omega_{ijk},
\end{displaymath}
while the right hand side is equal to
\begin{align*}
    f(\sum_i(x_i&+y_i)e_i) - f(\sum_i x_ie_i) - f(\sum_i y_ie_i)\\
    &=\sum_i(x_i+y_i)\omega_i + \sum_{i<j}(x_i+y_i)(x_j+y_j)\omega_{ij} + \sum_{i<j<k}(x_i+y_i)(x_j+y_j)(x_k+y_k)\omega_{ijk}\\
    &\quad-\sum_i x_i\omega_i - \sum_{i<j}x_ix_j\omega_{ij} - \sum_{i<j<k}x_ix_jx_k\omega_{ijk}\\
    &\quad-\sum_i y_i\omega_i - \sum_{i<j} y_iy_j\omega_{ij} - \sum_{i<j<k}y_iy_jy_k\omega_{ijk}.\\
\end{align*}
It can be seen easily that the two sides are equal.

Finally, since $f_3$ is trilinear and alternating, $f_4=0$.
\end{proof}

\subsection{The action of $GL(V)$ on the parameter space $\Omega_d$}\label{Ss:Action}

By Propositions \ref{Pr:Formulas} and \ref{Pr:Extension}, we can identity the space $F^V_4$ with the space of parameters
\begin{displaymath}
    \omega_i,\,\omega_{ij},\,\omega_{ijk}\in F,
\end{displaymath}
where $1\le i<j<k\le d$. We write the parameter space as
\begin{displaymath}
    \Omega_d = \bigoplus_I Fe_I,
\end{displaymath}
where the summation runs over all subsets $I=\{1,\dots,d\}$ such that $1\le |I|\le 3$. In particular, $\Omega_d$ has dimension $\binom{d}{1}+\binom{d}{2}+\binom{d}{3}$.

An element of $\Omega_d$ is written either as a sum $\sum_{I}\omega_Ie_I$, or as a tuple $(\omega_I)_I$, where the subsets $I$ are first ordered by cardinality and then lexicographically. For instance, an element of $\Omega_3$ is
\begin{displaymath}
    (\omega_1,\,\omega_2,\,\omega_3,\,\omega_{12},\,\omega_{13},\,\omega_{23},\,\omega_{123}).
\end{displaymath}

We now describe the action of $GL(V)$ on $\Omega_d$ that is equivalent to the natural action of $GL(V)$ on $F^V_4$.

\begin{proposition}\label{Pr:Action}
Let $V$ be a vector space over $F=\F_2$ with ordered basis $(e_1,\dots,e_d)$. Let $S\in GL(V)$ and $T=(t_{ij}) = S^{-1}$. Let $\omega\in \Omega_d$. The coordinates of $\omega^S$ are obtained as follows:
\begin{displaymath}
    \omega_{uvw}^S = \sum_{i<j<k}(t_{ui}t_{vj}t_{wk} + t_{ui}t_{vk}t_{wj} + t_{uj}t_{vi}t_{wk} + t_{uj}t_{vk}t_{wi} + t_{uk}t_{vi}t_{wj} + t_{uk}t_{vj}t_{wi})\omega_{ijk}
\end{displaymath}
for every $1\le u<v<w\le d$,
\begin{align*}
    \omega_{uv}^S = &\sum_{i<j}(t_{ui}t_{vj}+t_{uj}t_{vi})\omega_{ij}\\
    &\quad+ \sum_{i<j<k}(t_{ui}t_{uj}t_{vk} + t_{ui}t_{uk}t_{vj} + t_{uj}t_{uk}t_{vi} + t_{ui}t_{vj}t_{vk} + t_{uj}t_{vi}t_{vk} + t_{uk}t_{vi}t_{vj})\omega_{ijk}
\end{align*}
for every $1\le u<v\le d$, and
\begin{displaymath}
    \omega_u^S = \sum_i t_{ui}\omega_i + \sum_{i<j}t_{ui}t_{uj}\omega_{ij} + \sum_{i<j<k} t_{ui}t_{uj}t_{uk}\omega_{ijk}
\end{displaymath}
for every $1\le u\le d$.
\end{proposition}
\begin{proof}
Let $f\in F^V_4$ be the unique map such that \eqref{Eq:Values} holds for every $1\le i_1<i_2<i_3\le d$. Let $1\le u<v<w\le d$. Now \eqref{Eq:f3} implies that
\begin{displaymath}
    \omega_{uvw}^S = f_3^S(e_u,e_v,e_w) = f_3(e_uT,e_vT,e_wT) = f_3(\sum_i t_{ui}e_i,\sum_j t_{vj}e_j, \sum_k t_{wk}e_k)
\end{displaymath}
is equal to
\begin{displaymath}
     \sum_{i<j<k}(t_{ui}t_{vj}t_{wk} + t_{ui}t_{vk}t_{wj} + t_{uj}t_{vi}t_{wk} + t_{uj}t_{vk}t_{wi} + t_{uk}t_{vi}t_{wj} + t_{uk}t_{vj}t_{wi})\omega_{ijk},
\end{displaymath}
as claimed. The other two formulas follow analogously from \eqref{Eq:f2} and \eqref{Eq:f}, respectively.
\end{proof}

To understand the action of $GL(V)$ on $\Omega_d$, we decompose $\Omega_d$ as follows. For $1\le k\le 3$, let
\begin{displaymath}
    \Omega_d[k] = \sum_{|I|=k}F e_I,
\end{displaymath}
and for $\omega\in \Omega_d$ let $\omega[k]$ be the projection of $\omega$ onto $\Omega_d[k]$. Thus $\omega = \omega[1]\oplus\omega[2]\oplus\omega[3]$.

\begin{proposition}\label{Pr:ActionIsStratified}
Let $V$ be a vector space over $F=\F_2$ with ordered basis $(e_1,\dots,e_d)$, and let $S\in GL(V)$. For every $\omega\in \Omega_d$, there are square matrices
\begin{displaymath}
    N_1\in M_{\binom{d}{1}}(F),\quad N_2\in M_{\binom{d}{2}}(F),\quad
    N_3\in M_{\binom{d}{3}}(F)
\end{displaymath}
(which depend on $S$ but are independent of $\omega$) and vectors
\begin{displaymath}
    \nu_1(\omega[2],\omega[3])\in F^{\binom{d}{1}},\quad
    \nu_2(\omega[3])\in F^{\binom{d}{2}}
\end{displaymath}
(which depend on $S$ and the components of $\omega$ as indicated) such that
\begin{displaymath}
    \omega^S = (\omega[1]\oplus\omega[2]\oplus\omega[3])^S = (\omega[1]N_1+\nu_1(\omega[2],\omega[3]))
        \oplus(\omega[2]N_2 + \nu_2(\omega[3]))
        \oplus(\omega[3]N_3).
\end{displaymath}
\end{proposition}
\begin{proof}
This follows immediately from Proposition \ref{Pr:Action}. For instance, with $T=(t_{ij}) = S^{-1}$, the entry in row $ij$ and column $uv$ of $N_2$ is $t_{ui}t_{vj}+t_{uj}t_{vi}$, and the entry in column $uv$ of $\nu_2(\omega[3])$ is
\begin{displaymath}
    \sum_{i<j<k}(t_{ui}t_{uj}t_{vk} + t_{ui}t_{uk}t_{vj} + t_{uj}t_{uk}t_{vi} + t_{ui}t_{vj}t_{vk} + t_{uj}t_{vi}t_{vk} + t_{uk}t_{vi}t_{vj})\omega_{ijk}.
\qedhere
\end{displaymath}
\end{proof}

In particular, the following hold.
\begin{itemize}
\item The restriction of the action of $G=GL(V)$ onto $\Omega_d[3]$
    induces a linear action on $\Omega_d[3]$, namely $\omega[3]^S = \omega[3]N_3$.
\item For $A\in \Omega_d[3]$, let $G_A$ be the stabilizer of $A$ under the above action of $G$ on $\Omega_d[3]$. Then $G_A$ induces an affine action on $\Omega_d[2]\oplus A$, namely $(\omega[2]\oplus A)^S = (\omega[2]N_2+\nu_2(A))\oplus A$.
\item For $A\in \Omega_d[3]$ and $C\in \Omega_d[2]$, let $G_{C\oplus A}$ be the stabilizer of $C$ under the above action of $G_A$ on $\Omega_d[2]\oplus A$. Then $G_{C\oplus A}$ induces an affine action on $\Omega_d[1]\oplus C\oplus A$, namely $(\omega[1]\oplus C\oplus A)^S = (\omega[1]N_1+\nu_1(C,A))\oplus C\oplus A$.
\end{itemize}

\subsection{Stratified group actions}\label{Ss:Stratified}

For a group $G$ acting on a set $X$, let $X/G$ denote the set of all orbit representatives of $G$ on $X$. We now describe the orbit representatives of actions that behave
analogously to the action of $GL(V)$ on $\Omega_d$.

Let $X=X_1\times\cdots\times X_m$ be a set and suppose that a group $G$ acts on $X$. The action of $G$ on $X$ is \emph{stratified} (with respect to the decomposition $X_1\times\cdots\times X_m$) if:
\begin{enumerate}
\item[(i)] for every $1\le i\le m$ the action of $G$ on $X$ induces an action on $X_i\times\cdots\times X_m$, and
\item[(ii)] for every $1< i\le m$ and every $(x_i,\dots,x_m)\in X_i\times\cdots\times X_m$ the stabilizer $G_{(x_i,\dots,x_m)}$ induces an action on $X_{i-1}\times (x_i,\dots,x_m)$.
\end{enumerate}

\begin{proposition}\label{Pr:Stratified}
If the action of a group $G$ on $X=X_1\times\cdots\times X_m$ is stratified,
then $X/G$ consists of all tuples $(x_1,\dots,x_m)$, where
\begin{displaymath}
    x_m\in X_m/G,\quad
    x_{m-1}\in (X_{m-1}\times x_m)/G_{x_m},\quad\dots\quad,\quad
    x_1\in (X_1\times (x_2,\dots,x_m))/G_{(x_2,\dots,x_m)}.
\end{displaymath}
\end{proposition}
\begin{proof}
We prove the claim by induction on $m$. If $m=1$, we need to show that $X/G = X_1/G$, which is certainly true.

Suppose that $m=2$ and let $(y_1,y_2)\in X_1\times X_2$. There is a unique $x_2\in X_2/G$ such that $y_2^G=x_2^G$. Let $g\in G$ and $z_1\in X_1$ be such that $(y_1,y_2)^g = (z_1,x_2)$. There is a unique $x_1\in X_1/G_{x_2}$ such that $(z_1,x_2)^{G_{x_2}} = (x_1,x_2)^{G_{x_2}}$.

Finally, suppose that $m>2$ and the claim is true for $m-1$. With $X_2' = X_2\times\cdots\times X_m$, we see that the action of $G$ is also stratified with respect to the decomposition $X_1\times X_2'$. The result follows from the case $m=2$ and the inductive assumption.
\end{proof}

By Proposition \ref{Pr:ActionIsStratified}, the action of $GL(V)$ on the parameter space $\Omega_d$ is stratified with respect to the decomposition $\Omega_d[1]\oplus\Omega_d[2]\oplus\Omega_d[3]$. We deduce the following from Proposition
\ref{Pr:Stratified} and Theorem \ref{Th:AH}.

\begin{corollary}\label{Cr:Main}
Let $d\ge 3$ and let $V$ be a vector space over $\F_2$ with ordered basis $(e_1,\dots,e_d)$. Then $G=GL(V)$ acts on the parameter space $\Omega_d$ as in Proposition $\ref{Pr:Action}$, and the orbits of $G$ are in one-to-one correspondence with code loops of order $2^{d+1}$ up to isomorphism. Moreover, $\Omega_d/G$ consists of all vectors $w[1]\oplus w[2]\oplus w[3]$ such that
\begin{displaymath}
    w[3]\in \Omega_d[3]/G,\quad
    w[2]\in (\Omega_d[2]\oplus w[3])/G_{w[3]},\quad
    w[1]\in (\Omega_d[1]\oplus w[2]\oplus w[3])/G_{w[2]\oplus w[3]}.
\end{displaymath}
\end{corollary}

\subsection{Related actions}\label{Ss:RelatedActions}

Our enumeration is based on the usual action of $GL(V)$ on $F^V$, namely $\varphi(S)(f)=f^S$, $f^S(v) = f(vS^{-1})$. The reason for the inverse in the formula is best seen by considering a copy $W$ of $V$. If $S$ is a bijection $V\to W$ and $f\in F^V$, then $f^S\in F^W$ and we demand $f(v) = f^S(vS)$ for all $v\in V$, or, equivalently, $f^S(w) = f(wS^{-1})$ for all $w\in W$. By contrast, the enumeration of \cite{NagyVojtechovsky64} was based on the action $\psi(S)(f)(v) = f(vS^t)$, where $S^t$ is the transpose of $S$. We show that both actions yield the same orbits on $\Omega_d$ and hence the same code loops up to isomorphism.

\begin{lemma}\label{Lm:RelatedActions}
Let $G$ be a group and $\alpha$ an involutory automorphism of $G$. Let $\varphi$ be an action of $G$ on a set $X$, and let $\psi$ be the action of $G$ on $X$ defined by $\psi(g) = \varphi(g^\alpha)$. If $H\le G$ and $x\in X$,
then $\mathrm{Stab}(H,x,\varphi)^\alpha = \mathrm{Stab}(H^\alpha,x,\psi)$ and
$\mathrm{Orb}(H,x,\varphi) = \mathrm{Orb}(H^\alpha,x,\psi)$.
\end{lemma}
\begin{proof}
If $g\in\mathrm{Stab}(H,x,\varphi)$ then
$g^\alpha\in H^\alpha$ and $\psi(g^\alpha)(x) = \varphi(g)(x) = x$,
so $g^\alpha\in\mathrm{Stab}(H^\alpha,x,\psi)$.
Conversely, if $g^\alpha\in\mathrm{Stab}(H^\alpha,x,\psi)$ then
$g\in H$ and $\varphi(g)(x)=\psi(g^\alpha)(x)=x$,
so $g^\alpha\in \mathrm{Stab}(H,x,\varphi)^\alpha$.
Thus $y = \varphi(g)(x)$ for some $g\in H$ if and only if
$y=\psi(g^\alpha)(x)$ for some $g^\alpha\in H^\alpha$.
\end{proof}

Consider now the involutory automorphism $\alpha$ of $GL(V)$ given by $S^\alpha = S^{-t}$. Observe that
$\varphi(S^\alpha)(f)(v) = f(v(S^{-1})^{-t}) = f(vS^t) = \psi(S)(f)(v)$. Similar observations hold for the actions of $GL(V)$ on multivariate maps, which we also denote by $\varphi$ and $\psi$. Lemma \ref{Lm:RelatedActions} therefore applies, and we use it repeatedly.

If $A\in\Omega_d[3]$, then $H=\mathrm{Stab}(G,A,\varphi) = \mathrm{Stab}(G,A,\psi)^\alpha$ and $\mathrm{Orb}(G,A,\varphi) = \mathrm{Orb}(G,A,\psi)$.
If $C\in\Omega_d[2]$, then
$K = \mathrm{Stab}(H,C,\varphi) = \mathrm{Stab}(H^\alpha,C,\psi)^\alpha$ and $\mathrm{Orb}(H,C,\varphi) = \mathrm{Orb}(H^\alpha,C,\psi)$. Finally,
if $P\in\Omega_d[1]$, then $\mathrm{Stab}(K,P,\varphi) = \mathrm{Stab}(K^\alpha,P,\psi)^\alpha$ and $\mathrm{Orb}(K,P,\varphi) = \mathrm{Orb}(K^\alpha,P,\psi)$.

\section{The results}\label{Sc:Results}

\subsection{Code loops of given order}

\begin{table}[!ht]
\begin{displaymath}
    \begin{array}{|c|rrrrrrrrr|}
    \hline
    d       &0  &1  &2  &3  &4      &5      &6      &7      &8\\
    n       &2  &4  &8  &16 &32     &64     &128    &256    &512\\
    \hline
    \ell_n  &1  &2  &4  &10 &23     &88     &767    &80826  &937791557\\
    s_n     &1  &2  &4  &5  &7      &8      &10     &11     &13\\
    m_n     &1  &2  &5  &19 &122    &4529   &?      &?      &?\\
    g_n     &1  &2  &5  &14 &51     &267    &2328   &56092  &10494213\\
    \hline
    \end{array}
\end{displaymath}
\caption{Enumeration of certain classes of Moufang loops of order $n=2^{d+1}$ up to isomorphism}\label{Tb:Main}
\end{table}

Table \ref{Tb:Main} summarizes the results obtained by an algorithm based on Corollary \ref{Cr:Main}. For $n = 2^{d+1}$ with $0\le d\le 8$, Table \ref{Tb:Main} gives the number $\ell_n$ of code loops of order $n$ up to isomorphism and the number $s_n$ of small Frattini groups of order $n$ up to isomorphism. For comparison, we give the number $m_n$ of Moufang loops of order $n$ up to isomorphism and the number $g_n$ of groups of order $n$ up to isomorphism.


Only the numbers $\ell_{128}$, $\ell_{256}$, $\ell_{512}$ are new.
The numbers $g_n$ and $s_n$ are recorded in \cite{EickOBrien}.
(We can also obtain $s_n$ by applying our algorithm to vectors $\omega=\omega[1]\oplus\omega[2]\oplus\omega[3]$ with $\omega[3]=0$, that is, with trivial associator map.) It is known that all Moufang loops of order less than $12$ are associative \cite{Chein}. For $m_{16}$ and $m_{32}$, see \cite{Chein, GoodaireEtal}. We obtain $\ell_{16}$ and $\ell_{32}$ from
\cite{GoodaireEtal} or \texttt{LOOPS}. For $m_{64}$ and $\ell_{64}$, see \cite{NagyVojtechovsky64} or \texttt{LOOPS}.


\subsection{Code loops with a prescribed associator}

We now give more detailed computational results and also summarise
the techniques used in our computations. Most of the computations
were carried out using {\sf GAP} Version 4.7.9
on a computer with a 2.9 GHz processor.

For $n\le 128$, Corollary \ref{Cr:Main} can be routinely turned into an efficient algorithm. The running time for $n\le 64$ is in seconds, and for $n=128$ in minutes. For $n\ge 256$, the computational difficulties are considerable and additional improvements must be employed. We discuss the most difficult case
where $n=512$ and $d=8$, the case $256$ is similar.

There are $12$ trilinear alternating forms in
dimension $7$ \cite{CohenHelminck}, and $32$ in
dimension $8$ \cite{HoraPudlak}.
Hora and Pudl\'ak \cite{HoraPudlak} used radical polynomials and
other invariants to construct $32$ inequivalent forms in
dimension $8$, and then employed the Orbit-Stabilizer Theorem to
check that their list of forms is complete.

We used the trilinear alternating forms of \cite{HoraPudlak} and independently verified
their stabilizer claims. For each trilinear alternating form $A$ in dimension $8$,
we used randomized techniques, similar to those described in \cite[\S 8]{ELO},
to construct a subgroup $H_A$ of the stabilizer $G_A$ of $A$ under
the action of $G=GL(8,\F_2)$. We used
the explicit matrix action of $G$ in dimension $\binom{8}{3}=56$ as given
in Proposition \ref{Pr:Action}. If $G_A$ is small, then this technique sometimes
produced only a proper subgroup $H_A$ of $G_A$; if so, we
constructed the stabilizer of the form explicitly in $N_G (H_A)$, so
obtaining a larger $H_A$.  Finally, we verified that $\sum_{A} [G:H_A] = 2^{56}$,
which implies that $H_A =G_A$ for all $A$. This calculation was carried
out using {\sc Magma} \cite{Magma} and took about 7 days of CPU time.

For each $A$ and $G_A$, we then calculated the orbits of the action of $G_A$ on the
$\binom{8}{2}=28$-dimensional vector space $\Omega_8[2]$ (more precisely, on the
coset $\Omega_8[2]\oplus A$), by converting the affine action of
Proposition \ref{Pr:Action} into a permutation action. It took about $2$ hours per
generator to construct the permutations for $G_A$, and then a few minutes
to calculate the orbits.

Since we calculated the orbits on $\Omega_8[2]\oplus A$ explicitly, we knew the
stabilizer sizes and thus could employ randomized techniques to calculate the
stabilizer for each orbit representative $C$ of $(\Omega_8[2]\oplus A)/G_A$.

Finally, the orbits of $G_{C\oplus A}$ on $\Omega_8[1]\oplus C\oplus A$, a set of cardinality $2^8$, were calculated. The difficulty here lies in the number of choices $C\oplus A$ that needed to be considered. In the most extreme case, the stabilizer of $A$ has cardinality $192$ (compared to $|G|=5348063769211699200$), and there are $1424416$ choices for $C$, resulting in $359052160$ code loops with associator $A$. This
case took several hours to complete.

\begin{remark}\label{Rm:Feasible}
Currently, Hora and Pudl\'ak seek to
classify the trilinear alternating forms on $V=\F_2^9$.
They report that among the trilinear alternating forms on $V$ is one having trivial automorphism group.
By Corollary \ref{Cr:Main}, this form alone
contributes $2^{\binom{9}{1}+\binom{9}{2}} = 2^{45}$
pairwise non-isomorphic code loops of order $2^{10}$.
By contrast, there are $49\,487\,365\,422 \leq 2^{36}$ groups of
order $2^{10}$ \cite{SmallGroups}.
We expect that there are also forms with very small automorphism groups;
the associated orbit calculations
required in Corollary \ref{Cr:Main} may be infeasible.

\end{remark}

\begin{table}[!htbp]
\begin{scriptsize}
\begin{displaymath}
\begin{array}{|rrrrrrr|}
    \hline
    d/n     &ID & A                                             & \text{factors of }G_A         & |G_A|                 & C_A       & \ell_A\\
    \hline
    \hline
    3/16    &0  & \emptyset                                     & L_3(2)                        & 168                   & 2         & 5\\
            &1  & 123                                           & L_3(2)                        & 168                   & 2         & 5\\
    \hline
    4/32    &0  & \emptyset                                     & L_4(2)                        & 20160                 & 3         & 7\\
            &1  & 123                                           & L_2(7),\,2^3                  & 1344                  & 4         & 16\\
    \hline
    5/64    &0  & \emptyset                                     & L_5(2)                        & 9999360               & 3         & 8\\
            &1  & 123                                           & L_2(7),\,2^7,\,3              & 64512                 & 7         & 33\\
            &2  & 123{+}345                                     & A_6,\,2^5                     & 11520                 & 9         & 47\\
    \hline
    6/128   &0  & \emptyset                                     & L_6(2)                        & 20158709760           & 4         & 10\\
            &1  & 123                                           & L_2(7)^2,\,2^9                & 14450688              & 10        & 52\\
            &2  & 123{+}345                                     & A_6,\,2^{10}                  & 368640                & 22        & 174\\
            &3  & 123{+}456                                     & L_2(7)^2,\,2                  & 56448                 & 20        & 224\\
            &4  & 123{+}345{+}156                               & L_2(7),\,2^8                  & 43008                 & 19        & 234\\
            &5  & 123{+}234{+}345{+}246{+}156                   & L_3(4),\,2,\,3                & 120960                & 10        & 73\\
    \hline
    7/256   &0  & \emptyset                                     & L_7(2)                        & 163849992929280       & 4         & 11\\
            &1  & 123                                           & L_2(7),\,A_8,\,2^{12}         & 13872660480           & 13        & 72\\
            &2  & 123{+}345                                     & A_6,\,2^{16},\,3              & 70778880              & 40        & 381\\
            &3  & 123{+}456                                     & L_2(7)^2,\,2^7                & 3612672               & 53        & 903\\
            &4  & 123{+}345{+}156                               & L_2(7),\,2^{14}               & 2752512               & 57        & 968\\
            &5  & 123{+}234{+}345{+}246{+}156                   & L_3(4),\,2^7,\,3              & 7741440               & 23        & 269\\
            &6  & 123{+}345{+}567                               & 2^{13},\,3^2                  & 73728                 & 289       & 10019\\
            &7  & 123{+}145{+}167{+}357                         & L_2(7),\,2^{12}               & 688128                & 69        & 1459\\
            &8  & 123{+}167{+}246{+}357                         & 2^{10},\,3^2                  & 9216                  & 634       & 39916\\
            &9  & 123{+}145{+}167                               & Sp_6(2),\,2^6                 & 92897280              & 23        & 167\\
            &10 & 123{+}145{+}167{+}246{+}357                   & U_3(3),\,2                    & 12096                 & 324       & 25052\\
            &11 & 123{+}234{+}345{+}246{+}156{+}367             & A_5,\,2^{11},\,3              & 368640                & 67        & 1609\\
    \hline
    8/512   &0  & \emptyset                                     & L_8(2)                        & 5348063769211699200   & 5         & 13\\
            &1  & 123                                           & L_2(7),\,L_5(2),\,2^{15}      & 55046716784640        & 16        & 92\\
            &2  & 123{+}345                                     & L_2(7),\,A_6,\,2^{20}         & 63417876480           & 59        & 627\\
            &3  & 123{+}456                                     & L_2(7)^2,\,2^{14},\,3         & 1387266048            & 104       & 2040\\
            &4  & 123{+}345{+}156                               & L_2(7),\,2^{21},\,3           & 1056964608            & 110       & 2181\\
            &5  & 123{+}234{+}345{+}246{+}156                   & L_3(4),\,2^{14},\,3^2         & 2972712960            & 46        & 603\\
            &6  & 123{+}345{+}567                               & 2^{20},\,3^2                  & 9437184               & 910       & 42058\\
            &7  & 123{+}145{+}167{+}357                         & L_2(7),\,2^{19}               & 88080384              & 213       & 6157\\
            &8  & 123{+}167{+}246{+}357                         & 2^{17},\,3^2                  & 1179648               & 1968      & 162636\\
            &9  & 123{+}145{+}167                               & Sp_6(2),\,2^{13}              & 11890851840           & 59        & 655\\
            &10 & 123{+}145{+}167{+}246{+}357                   & U_3(3),\,2^8                  & 1548288               & 978       & 100396\\
            &11 & 123{+}234{+}345{+}246{+}156{+}367             & A_5,\,2^{18},\,3              & 47185920              & 201       & 6588\\
            &12 & 123{+}345{+}678                               & L_2(7),\,A_6,\,2^5            & 1935360               & 942       & 76858\\
            &13 & 123{+}145{+}178{+}246                         & 2^{18},\,3^2                  & 2359296               & 1175      & 72552\\
            &14 & 123{+}145{+}268{+}347                         & 2^{11},\,3^2                  & 18432                 & 25352     & 4553608\\
            &15 & 123{+}345{+}567{+}178                         & 2^{15},\,3^2                  & 294912                & 3121      & 340812\\
            &16 & 123{+}145{+}168{+}246{+}257                   & 2^{17},\,3                    & 393216                & 2718      & 269244\\
            &17 & 123{+}145{+}168{+}347{+}256                   & 2^{10},\,3                    & 3072                  & 108136    & 24014336\\
            &18 & 123{+}145{+}168{+}347{+}267                   & 2^{14},\,3                    & 49152                 & 9050      & 1597720\\
            &19 & 123{+}145{+}278{+}356{+}467                   & 2^8,\,3^2                     & 2304                  & 129180    & 30780784\\
            &20 & 123{+}145{+}178{+}246{+}258{+}347             & 2^{13},\,3                    & 24576                 & 14252     & 2962796\\
            &21 & 123{+}145{+}168{+}347{+}258{+}267             & 2^6,\,3                       & 192                   & 1424416   & 359052160\\
            &22 & 123{+}145{+}257{+}278{+}368{+}467             & L_2(7),\,2                    & 336                   & 808692    & 204763400\\
            &23 & 123{+}145{+}168{+}246{+}257{+}356{+}456       & A_5,\,2^{13},\,3              & 1474560               & 718       & 65885\\
            &24 & 123{+}145{+}257{+}258{+}268{+}348{+}467       & L_2(8),\,2^6,\,3              & 96768                 & 3392      & 732448\\
            &25 & 123{+}145{+}167{+}178{+}258{+}267{+}347{+}356 & 2^9,\,3,\,7                   & 10752                 & 25952     & 6424768\\
            &26 & 123{+}145{+}178{+}246{+}258{+}347{+}356{+}456 & 2^4,\,3^3                     & 432                   & 628452    & 159271112\\
            &27 & 123{+}145{+}258{+}356{+}478{+}567             & 2^8,\,3                       & 768                   & 381912    & 92169184\\
            &28 & 123{+}145{+}167{+}246{+}257{+}267{+}368       & 2^{10},\,3                    & 3072                  & 99452     & 23205904\\
            &29 & 123{+}145{+}246{+}468{+}578                   & 2^{10},\,3                    & 3072                  & 109276    & 24331744\\
            &30 & 123{+}145{+}258{+}347{+}368{+}567             & A_5,\,2^8,\,3                 & 46080                 & 9132      & 1686096\\
            &31 & 123{+}145{+}457{+}678                         & 2^{10},\,3^4                  & 82944                 & 6982      & 1096100\\
        \hline
\end{array}
\end{displaymath}
\end{scriptsize}
\caption{The number $\ell_A$ of code loops with associator map $A$}\label{Tb:CLA}
\end{table}

The detailed results for $16\le n=2^{d+1} \le 512$ are summarized in Table \ref{Tb:CLA}. In the first column we give the dimension $d$ of the underlying vector space over $\F_2$ and the order $n=2^{d+1}$ of the resulting code loops, in the second column we list the ID for a trilinear alternating form $A$ on $\F_2^d$ (an eventual associator map), in the third column we list the trilinear alternating form $A$ (our numbering and choice of basis follow those of \cite{HoraPudlak}), in the fourth column we list composition factors, with multiplicities,
of the stabilizer $G_A$ of $A$ in $GL(d,\F_2)=L_d(2)$ (using the standard notation of \cite{Atlas}), in the fifth column we give the order of $G_A$, in the sixth column we give the number of orbits of $G_A$ on $\Omega_d[2]\oplus A$ (eventual commutator maps), and in the last column we give the number $\ell_A$ of code loops with associator $A$ up to isomorphism.

If $d\le 7$ or $|G_A|>10000$,
then we stored all orbit representatives of the action
of $GL(d,\F_2)$ on $\Omega_d$, from which the code loops can be explicitly constructed
using the method of \cite{Nagy}. In the remaining situations, we only counted
the number of representatives. More detailed data files are available on request from the second author.

\subsection{Specific code loops and their automorphism groups}

We conclude by commenting on specific code loops and their automorphisms.

\begin{lemma}\label{Lm:AuxMoufang}
Let $Q$ be a Moufang loop and let $Z$ be a cyclic central subloop of $Q$ such that $Q/Z$ is at most $2$-generator. Then $Q$ is a group.
\end{lemma}
\begin{proof}
Let $Z=\langle z\rangle$ and $Q/Z=\langle xZ,yZ\rangle$. Then $Q=\langle x,y,z\rangle$ and $A(x,y,z)=1$ because $z$ is central. By Moufang's Theorem
\cite{Moufang}, $Q$ is a group.
\end{proof}

\begin{lemma}\label{Lm:16}
Every nonassociative Moufang loop of order $16$ is a code loop.
\end{lemma}
\begin{proof}
Recall that every Moufang $2$-loop $Q$ is centrally nilpotent \cite{GlaubermanWright}. If $|Q|\le 8$ then there is $Z\le Z(Q)$ of order $2$ such that $|Q/Z|\le 4$, so $Q/Z$ is at most $2$-generator and $Q$ is a group by Lemma \ref{Lm:AuxMoufang}.

Suppose now that $|Q|=16$ and $Q$ is not associative. Let $Z\le Z(Q)$ have order $2$. Then $Q/Z$ is a group of order $8$ by the first paragraph, and $Q/Z$ is not generated by any two of its elements by Lemma \ref{Lm:AuxMoufang}. Hence $Q/Z$ is the elementary abelian group of order $8$.
\end{proof}

There are $5$ nonassociative code loops of order $16$ according to Table \ref{Tb:CLA}, and thus precisely $5$ nonassociative Moufang loops of order $16$ by Lemma \ref{Lm:16}. This agrees with the classification of \cite{Chein}.
The automorphism groups of these $5$ loops are as follows:
soluble groups of order $2^4\,3$, $2^6$ and $2^6\,3$ (twice);
and a group of order $1344$ with composition factors $L_2(7)$, $2^3$. The last loop is the \emph{octonion loop} $\mathbb O_{16}$ that captures the multiplication rules among the eight basic octonionic units and their additive inverses.

There are $71$ nonassociative Moufang loops of order $32$; among them are $16$
code loops. The unique such code loop with the largest automorphism group (and
composition factors $L_2(7)$, $2^7$) is the direct product of
$\mathbb O_{16}$ with the cyclic group of order $2$.

The most famous code loop, the \emph{Parker loop} $\mathcal P$, is obtained as the code loop of the extended binary Golay code $\mathcal G_{24}$ of dimension $12$ (and hence is not found in our classification). The automorphism group of $\mathcal G_{24}$ is the Mathieu group $M_{24}$, and the group of so-called \emph{standard} automorphisms of $\mathcal P$ (those automorphisms that belong to $M_{24}$ when signs of elements of $\mathcal P$ are ignored) has structure $2^{12}\cdot M_{24}$ \cite[Chapter 29]{ConwaySloane}.

A potentially rewarding
research program is to investigate nonassociative
code loops with ``large" or ``interesting" automorphism groups.

\section*{Acknowledgment}

Petr Vojt\v{e}chovsk\'y thanks the Department of Mathematics, University of Auckland,
for its
hospitality and for access to its high performance computing facility.

The {\sf GAP} code used in the implementation of our algorithm, although completely
rewritten and highly
optimized for speed, originated from the code of \cite{NagyVojtechovsky64}.

We thank G\'abor Nagy, the referee and editor for helpful comments and suggestions.

\end{document}